\documentclass[12pt,reqno]{amsart}
\usepackage{fullpage,graphicx,amsfonts,amssymb,amsmath,amsthm,mathrsfs,mathtools}
\usepackage{longtable}
\usepackage[all]{xy}
\usepackage{color} 
\usepackage{hyperref}
\usepackage{enumerate}

\theoremstyle{plain} 
\newtheorem{theorem}    {Theorem}

\newtheorem{lemma}      [theorem]{Lemma}
\newtheorem{corollary}  [theorem]{Corollary}
\newtheorem{proposition}[theorem]{Proposition}

\theoremstyle{definition}

\theoremstyle{remark}
\newtheorem{remark}              {Remark}

\newcommand{\Q}{\mathbb Q}
\newcommand{\Z}{\mathbb Z}

\newcommand{\C}{\mathbb C}

\newcommand{\floor}[1]{\lfloor #1 \rfloor}
\newcommand{\Mod}[1]{\ (\text{mod}\ #1)}

\newcommand{\SL}{\operatorname{SL}}

\begin{document}

\title{The Trace of $T_2$ takes no repeated values}

\author{Liubomir Chiriac \and Andrei Jorza}


\address{Portland State University, Fariborz
  Maseeh Department of Mathematics and Statistics, Portland, OR
  97201}
\email{chiriac@pdx.edu}
\address{University of Notre Dame, 275 Hurley Hall, Notre Dame, IN 46556}\email{ajorza@nd.edu}


\begin{abstract}

We prove that the trace of the Hecke operator $T_2$ acting on the vector space of cusp forms of level one takes no repeated values, except for 0, which only occurs when the space is trivial. 

\end{abstract}


\maketitle

\section{Introduction}

Given two normalized cuspidal eigenforms of same level $N$ but different weights, according to a result of Ghitza \cite{Ghitza}, there exists an index $n\leq 4(\log N+1)^2$ where the corresponding Fourier coefficients of the two forms differ. A consequence of this fact is that a newform of level one is determined by the first four Fourier coefficients. Numerical experiments suggest that the threshold can further be improved to $n=2$. In other words, to distinguish two newforms of level one it is enough to compare only their second Fourier coefficients. Since newforms are normalized so that their first coefficients are one, $n=2$ is clearly the optimal bound in this case. 
 
In fact, Vilardi and Xue \cite{Vilardi-Xue} have shown that the
above statement holds under the assumption of Maeda's conjecture
for the Hecke operator $T_2$. A key observation in
\cite{Vilardi-Xue} is that the trace of $T_2$ does not take the
same value at two distinct weights that are close to each
other. It has been conjectured that this phenomenon happens in general, for any two distinct weights (see \cite[Conjecture 3.7]{Xue-Zhu} and \cite[Remark 3.6]{Vilardi-Xue}). 

The purpose of this article is to prove the aforementioned conjecture. Throughout the paper, we denote by ${\rm Tr ~T_2}(2k,\SL_2(\Z))$ the trace of the Hecke operator $T_2$ acting on the space $\mathcal{S}_{2k}(\SL_2(\Z))$ of cusp forms of weight $2k$ and level one. Recall that $\dim_{\C} \mathcal{S}_{2k}(\SL_2(\Z))=0$ precisely when $k<6$ or $k=7$. In particular, the trace of $T_2$ is trivially zero in those cases. 

\begin{theorem} There are no repeated values in the sequence $\{ {\rm Tr ~T_2}(2k,\SL_2(\Z)) \}_{k\geq 2}$, except for 0, which only occurs trivially when $k<6$ or $k=7$.
\end{theorem}

Our strategy consists of three main steps. First, using the Eichler--Selberg trace formula, we show that the trace is given by a recurrence relation of degree two. Second, we employ $p$-adic methods to reduce the range of possible coincidences for the trace of $T_2$, using an explicit formula for its 2-adic valuation, which was originally established by the authors in \cite{chiriac-jorza:newton} in a more general context. Finally, we apply effective results from the theory of linear forms in two logarithms to prove that the recurrence has no repeated values. 

\section{The recurrence satisfied by the trace of \texorpdfstring{$T_2$}{T\_2}}
It is convenient to express the trace of $T_2$ in terms of a simple sequence which satisfies a linear recurrence of order two.

Consider the sequence $\{a_n\}_{n\geq 0}$ starting with $a_0=1$, $a_1=-1$, and for all integers $n\geq 2$:
\begin{equation} \label{def_a_n}
a_n=-3a_{n-1}-4 a_{n-2}.
\end{equation}
In this section we express ${\rm Tr ~T_2}(2k,\SL_2(\Z))$ in terms of the sequence $\{a_n\}$.  

\begin{proposition} \label{a(n)}
Let  $k\geq 2$ be an integer. Then 
	\begin{equation} \label{trace_T2}
		{\rm Tr ~T_2}(2k,\SL_2(\Z)) =
			\begin{cases}
				-a_{k-1}-1, \text{ if }k\equiv 2,3\Mod 4 \\
				-a_{k-1}-1+2^{k-1}, \text { if } k\equiv 0\Mod 4 \\
				-a_{k-1}-1-2^{k-1}, \text { if } k\equiv 1\Mod 4.
			\end{cases}
	\end{equation}
\end{proposition}

Proposition~\ref{a(n)} is an application of the Eichler--Selberg trace formula. Before stating it, we first introduce some notations.

The \textit{Hurwitz class number} $H(n)$ is the number of equivalence classes with respect to $\SL_2(\Z)$ of positive definite binary quadratic forms of discriminant $-n$. We count the class containing $x^2+y^2$ with multiplicity $1/2$, and the class containing $x^2+xy+y^2$ with multiplicity $1/3$. By convention $H(0)=-1/12$ and $H(n)=0$ if $n<0$. Moreover, we note that $H(n)=0$ whenever $n\equiv 1$ or $2\Mod 4$.

For convenience, we record below the first few nonzero values of $H(n)$:
{\renewcommand{\arraystretch}{2}
\begin{center}
\begin{tabular}{ c | c c c c c c c c c c c c c c c c c}
$n$ & 0 & 3 & 4 & 7 & 8 & 11 & 12 & 15 & 16 & 19 & 20 & 23 & 24 & 27 & 28 & 31 & 32\\  \hline
$H(n)$ & $-\frac{1}{12}$ & $\frac{1}{3}$ & $\frac{1}{2}$ & 1 & 1 & 1 & $\frac{4}{3}$ & 2 & $\frac{3}{2}$ & 1 & 2 & 3 & 2 & $\frac{4}{3}$ & 2 & 3 & 3

\end{tabular}
\end{center} }

Let $k\geq 2$ be an integer. Another object that appears in the formula is the polynomial $P_{2k}(t,n)$, defined as the coefficient of $x^{2k-2}$ in the power series expansion of $(1-tx+nx^2)^{-1}$. 

With the above notation, the Eichler--Selberg trace formula (as given in the Appendix by Zagier to \cite{Zag}) reads: 

$${\rm Tr ~T}_n (2k,\SL_2(\Z)) =-\frac{1}{2}\sum_{|t|\leq 2\sqrt{n}} P_{2k}(t,n) H(4n-t^2)-\frac{1}{2}\sum_{dd'=n} \min(d,d')^{2k-1}.$$

Specializing at $n=2$ we prove the following combinatorial formula.  

\begin{lemma} \label{comb}
Let $k\geq 2$ be an integer. Then
$${\rm Tr ~T_2}(2k,\SL_2(\Z))=(-2)^{k-2}-1-\sum_{j=0}^{k-1}(-1)^j \binom{2k-2-j}{j}2^j(1+2^{2k-3-2j}).$$
\end{lemma}

\begin{proof}
Given a power series $A(x)$, we denote by $[x^n]A(x)$ the coefficient corresponding to $x^n$. This allows us to write $P_{2k}(t,n)$ as follows,
\begin{align*}
P_{2k}(t,n)&=[x^{2k-2}]\frac{1}{1-tx+nx^2} \\
&=[x^{2k-2}] \sum_{m\geq 0} \left( x(t-nx) \right)^m \\
&=\sum_{m= k-1}^{2k-2} [x^{2k-2-m}] \sum_{j=0}^m \binom{m}{j}(-nx)^j t^{m-j} \\
&=\sum_{m= k-1}^{2k-2} \binom{m}{2k-2-m} (-n)^{2k-2-m} t^{2m-(2k-2)} .
\end{align*}

Setting $j=2k-2-m$ we obtain 
$$P_{2k}(t,n)=\sum_{j=0}^{k-1} (-1)^j \binom{2k-2-j}{j} n^j t^{2k-2-2j}.$$

In particular, $P_{2k}(t,n)=P_{2k}(-t,n)$. For $n=2$ the trace formula yields 
\begin{align*}
{\rm Tr ~T_2}(2k,\SL_2(\Z))&=-\frac{1}{2} P_{2k}(0,2)H(8)-P_{2k}(1,2)H(7)-P_{2k}(2,2)H(4)-1 \\
&=-\frac{1}{2}P_{2k}(0,2)-P_{2k}(1,2)-\frac{1}{2}P_{2k}(2,2)-1. 
\end{align*} 
Using the above explicit description of $P_{2k}(t,n)$ we obtain the desired expression for the trace of $T_2$. 
\end{proof}

\begin{lemma} \label{parts}
	\begin{enumerate}[(i)]
		\item The terms of the sequence $\{a_n\}_{n\geq 0}$ defined in \eqref{def_a_n} satisfy $$a_n=\sum_{j=0}^n (-1)^j \binom{2n-j}{j} 2^j.$$
		\item For all integers $n\geq 0$ we have $$\sum_{j=0}^n (-1)^j \binom{2n-j}{j} 2^{n-j}=(-1)^{\floor{n/2}}.$$
	\end{enumerate}
\end{lemma}

\begin{proof} (i) The generating function defined by the sequence on the right-hand side is 
$$\sum_{n\geq 0} \sum_{j=0}^n (-1)^j \binom{2n-j}{j} 2^j x^n.$$
With the change of variables $m=n-j$ we get 
\begin{align*}
\sum_{m\geq 0} \sum_{j \geq 0} (-1)^j \binom{2m+j}{j} (2x)^j x^m
&=\sum_{m\geq 0} \frac{x^m}{(1+2x)^{2m+1}} \\
&=\frac{1}{1+2x}\sum_{m\geq 0} \left( \frac{x}{(1+2x)^2} \right)^m \\
&=\frac{1}{1+2x} \left( 1- \frac{x}{(1+2x)^2} \right) ^{-1} \\
&=\frac{1+2x}{1+3x+4x^2},
\end{align*} where in the first and third equalities we have used the negative binomial series (for integers $-d<0$) $$(1+x)^{-d}=\sum_{j\geq 0} (-1)^j \binom{d+j-1}{j}x^j.$$ At the same time, the generating function $A(x)=\sum_{n\geq 0}a_nx^n$ of the sequence $\{a_n\}_{n\geq 0}$ can be found from the identity 
$$\sum_{n\geq 2} \left( a_n+3a_{n-1}+4a_{n-2} \right) x^n=0,$$ which is equivalent to
$$(A(x)-1+x)+3x(A(x)-1)+4x^2A(x)=0.$$ This shows that $$A(x)=\frac{1+2x}{1+3x+4x^2},$$ so the two sequences that appear in part (i) define the same generating function. 

\medskip

(ii) A similar computation shows that 
$$\sum_{n\geq 0} \sum_{j=0}^n (-1)^j \binom{2n-j}{j} 2^{n-j} x^n=\frac{1+x}{1+x^2}=\sum_{n\geq 0} (-1)^{\floor{n/2}} x^n.$$
\end{proof}

We are now ready to prove Proposition~\ref{a(n)}.

\begin{proof}[Proof of Proposition~\ref{a(n)}] 

By combining Lemma~\ref{comb} and Lemma~\ref{parts} (for $n=k-1$) we see that
\begin{align*}
{\rm Tr ~T_2}(2k,\SL_2(\Z))&=(-2)^{k-2}-1-\sum_{j=0}^{k-1}(-1)^j \binom{2k-2-j}{j}2^j(1+2^{2k-3-2j}) \\
&=(-2)^{k-2}-1- a_{k-1} - (-1)^{\floor{(k-1)/2}}2^{k-2} \\
&=-a_{k-1}-1+2^{k-2}\left( (-1)^k+(-1)^{\floor{(k+1)/2}} \right),
\end{align*}
and the conclusion follows by considering $k$ modulo 4. 
\end{proof}

\section{The 2-adic valuation of the trace of \texorpdfstring{$T_2$}{T\_2}}

In this section we give an explicit formula for the 2-adic valuation $v_2$ of the trace of $T_2$. A more general formula, concerning the trace of $\wedge^n T_2$ for $n\leq 15$, appears in \cite[Theorem 12]{chiriac-jorza:newton}. 

\begin{proposition} \label{2-adic_valuation_Tr_T_2} For integers $k\gg 0$ we have
$$v_2 \left( {\rm Tr ~T_2}(2k,\SL_2(\Z)) \right)=3+v_2\left(k-\Omega \right),$$ where 
$$\Omega=\frac{1}{2} \left( \frac{\log_2(1-2\omega)}{\log_2(1-\omega)}+1 \right)\in \Z_2.$$ Here the function $\log_2:1+2\Z_2\to 2\Z_2$ is the $2$-adic logarithm, and $\omega=\frac{1-\sqrt{-7}}{2}$ is embedded in $\mathbb{Q}_2$ so that $v_2(\omega)=1$. 
\end{proposition}
As stated, Proposition \ref{2-adic_valuation_Tr_T_2} is insufficient for our purposes.
By Proposition \ref{a(n)}, $${\rm Tr ~T_2} (2k,\SL_2(\mathbb{Z}))=-a_{k-1}-1+\varepsilon 2^{k-1},$$ where
$\varepsilon=0,\pm 1$. As the sequence $\{ a_n \}$ is always odd,
the trace of $T_2$ is always even, but its valuation is hard to
pin down except when $k\gg 0$, because one cannot control the
2-adic expansion of $\Omega$. The strategy of \cite[Theorem
12]{chiriac-jorza:newton} is to write explicitly
\begin{equation}\label{trT2}
{\rm Tr ~T_2} (2k,\SL_2(\mathbb{Z}))=\left(\frac{\overline{\omega}^{2k-1}}{\omega-\overline{\omega
  }}-1\right)-\frac{\omega^{2k-1}}{\omega-\overline{\omega
  }}+\varepsilon 2^{k-1},
\end{equation}
where $\overline{\omega}=\frac{1+\sqrt{-7}}{2}$ is the conjugate
of $\omega$, and to compute each valuation separately. As $k\gg
0$ the valuation of the first term dominates. To account for this ineffective condition we need to prove a slightly more general
result.
\begin{lemma}\label{l:v2}Suppose $t$ is a multiple of 8. Then
\[v_2\left(\frac{\overline{\omega}^{2k-1}}{\omega-\overline{\omega
  }}-1-t\right)=3+v_2(k-\Omega_t),\]
where 
$$\Omega_t=\frac{1}{2} \left(\frac{\log_2((1+t)(2\omega-1))}{\log_2(1-\omega)}+1 \right)\in \Z_2.$$

\end{lemma}

\begin{proof}
Note that $v_2(\overline{\omega})=v_2(1-\omega)=0$ and
$\omega-\overline{\omega}=-(1-2 \omega)$ is a unit. Therefore
\begin{align*}
v_2\left(\frac{\overline{\omega}^{2k-1}}{\omega-\overline{\omega
  }}-1-t\right)&=v_2\left(\overline{\omega}^{2k-1}+(1+t)(1-2\omega)\right)\\
  &=v_2\left(\overline{\omega}^{2k-1}-(1-\omega)^{2
    \Omega_t-1}\right)\\
  &=v_2\left((1-\omega)^{2(k-\Omega_t)}-1\right)\\
  &=v_2((1-\omega)^2-1)+v_2(k-\Omega_t)=3+v_2(k-\Omega_t)
\end{align*}
by the Lifting the Exponent Lemma. We remark that it is in the second equality that we use that $t$ is a multiple of 8, for convergence reasons.
\end{proof}
From Lemma \ref{l:v2} and equation \eqref{trT2} we immediately obtain that if $t$ is a multiple of 8 then
\[v_2({\rm Tr ~T_2} (2k,\SL_2(\mathbb{Z}))-t)\geq \min(k-1,3+v_2(k-\Omega_t)),\]
with equality whenever $v_2(k-\Omega_t)<k-4$.
This hypothesis is satisfied for all $k\gg 0$ (see \cite[Proposition 7]{chiriac-jorza:newton}), therefore yielding a proof of Proposition \ref{2-adic_valuation_Tr_T_2}. In order to obtain an effective result, we apply Lemma \ref{l:v2} directly and obtain the following:
\begin{corollary}\label{c}
If ${\rm Tr ~T_2}(2k,\SL_2(\mathbb{Z}))={\rm Tr ~T_2}(2k',\SL_2(\mathbb{Z}))$ for $k,k'\geq 6$ then $k\equiv k'\pmod{4}$.
\end{corollary}
\begin{proof}
As $\Omega\equiv 3\pmod{4}$ and $\Omega_8\equiv 2\pmod{4}$ we see that
\begin{center}
\begin{tabular}{lll}
  $k\mod 4$ & $v_2(k-\Omega)$ & $v_2(k-\Omega_8)$\\
  \hline
  0 & 0 & 1\\
  1 & 1 & 0\\
  2 & 0 & $\geq 2$\\
  3 & $\geq 2$& 0
\end{tabular}
\end{center}
If the two traces are equal, the valuations of ${\rm Tr ~T_2}(2k,\SL_2(\Z))-t$ and ${\rm Tr ~T_2}(2k',\SL_2(\Z))-t$ must be
equal for all $t$. We have $v_2({\rm Tr ~T_2}(2k,\SL_2(\mathbb{Z})))\geq \min(k-1,3+v_2(k-\Omega))$ with equality if $k\not\equiv 3\pmod{4}$, as then
$v_2(k-\Omega)\leq 1<k-4$. Similarly, $v_2({\rm Tr ~T_2}(2k,\SL_2(\mathbb{Z}))-8)\geq \min(k-1,3+v_2(k-\Omega_8))$ with equality if
$k\not\equiv 2\pmod{4}$, as then $v_2(k-\Omega_8)\leq 1<k-4$.

The desired congruence now follows from a case-by-case analysis using the above table.
\end{proof}

We conclude that if ${\rm Tr
  ~T_2}(2k,\SL_2(\mathbb{Z}))=-a_{k-1}-1+\varepsilon 2^{k-1}$
has a repeated value, then one of the following sequences has a
repeated value: $\{a_n \}$, $\{ a_n+2^n\}$, $\{ a_n-2^n\}$. Technically, Corollary \ref{c} only applies to repeated values when $k\geq 6$. However, as ${\rm Tr
  ~T_2}(2k,\SL_2(\mathbb{Z}))=0$ for $k<6$ and $k=7$, this conclusion applies to all $k\geq 2$.

We will treat each of these cases separately.

\section{Repeated values in the sequence \texorpdfstring{$\{a_n\}$}{a\_n}}

The first few terms of the sequence $\{a_n\}_{n\geq 0}$ defined in \eqref{def_a_n} are 
$$1,-1,-1,7,-17,23,-1,-89,271,-457,287,\ldots$$
The purpose of this section is to show that, except for the three occurrences of the value -1, there are no other repeated values. Our main tool is the method of linear forms in logarithms. 

\bigskip

If $\alpha$ is an algebraic number with minimal polynomial $$a_mx^m+\ldots +a_0=a_m(x-r_1)\ldots(x-r_m)\in \Z[x]$$ then its (logarithmic) height is given by 
$$h(\alpha)=\frac{1}{m} \left( \log |a_m|+\sum_{j=1}^m \max (0, \log |r_i| ) \right).$$

One of the sharpest estimates for linear forms in two logarithms is due to Laurent, Mignotte and Nesterenko; we recall here one of their results (\cite[Th\'eor\'eme 3]{LMN}).

\begin{lemma} \label{sharp_two_log}
Let $\alpha$ be an algebraic number, which is not a root of unity and $|\alpha|=1$. Let $b_1$ and $b_2$ be a positive integers, and 
$$\Lambda=b_1i \pi -b_2\log \alpha.$$ Set $D=[\Q(\alpha):\Q]/2$, $a =\max \{ 20, 10.98  |\log \alpha |+ Dh(\alpha) \}$, and
\begin{align*}
H &=\max \{ 17, \frac{\sqrt{D}}{10}, \log \left( \frac{b_1}{2a}+\frac{b_2}{68.9} \right) +2.35D+5.03 \}.
\end{align*} Then $$\log |\Lambda |\geq -8.87 aH^2.$$
\end{lemma}

\bigskip

\begin{remark} \label{sharp_variant}
We will be concerned with algebraic numbers $\alpha$ of degree 2, i.e., those for which $D=1$. In this setting, an easy consequence of the above result (see \cite[Theorem 2.6]{Bugeaud} for a proof) is that for all positive integers $n$ we have
$$\log|\alpha^n-1| \geq -9aH_1^2,$$
where $H_1=\max\{ 17,\log \left( \frac{n}{25} \right)+7.45\}$. 
\end{remark}

\begin{proposition} \label{compare_a_n}
Consider the sequence $\{a_n\}$ defined in \eqref{def_a_n}. There are no two distinct integers $5\leq n<m$ such that $a_n=a_m.$ In particular, as $k\geq 6$ ranges over the integers congruent to 2 or 3 modulo 4, all the values of ${\rm Tr ~T_2}(2k,\SL_2(\Z))$ are distinct. 
\end{proposition}


\begin{proof}
Since $a_n=\frac{1}{\omega-\overline{\omega}} (\omega^{2n+1}-\overline{\omega}^{2n+1} )$ with $\omega=\frac{1-\sqrt{-7}}{2}$, it is enough to show that the equation 
$$\omega^n-\overline{\omega}^n=\omega^m-\overline{\omega}^m$$ has no integer solutions $11\leq n<m$. Assume, by contradiction, that such a solution exists. Then we consider the 2-adic valuation of both sides of the identity
\begin{equation} \label{ident_omega_omega_bar}
\omega^n(\omega^{m-n}-1)=-\overline{\omega}^n (1-\overline{\omega}^{m-n}).
\end{equation}
On the one hand, since $v_2(\omega)=1$, it is clear that 
$$v_2 \left( \omega^n(\omega^{m-n}-1) \right) = v_2(\omega^n)+v_2(\omega^{m-n}-1)\geq n. $$
On the other hand, $v_2(\overline{w})=0$ and $v_2(1- \overline{w})=1$, so
\begin{align*}
v_2 \left( -\overline{\omega}^n(1-\overline{\omega}^{m-n}) \right) &= v_2(1-\overline{\omega}^{m-n}) \\
&\leq v_2(m-n)+2 \\ 
& \leq \frac{ \log(m)}{ \log(2)}+2.
\end{align*}
Hence (\ref{ident_omega_omega_bar}) leads to the following inequality
$$n \leq \frac{\log(m)}{\log(2)}+2.$$ In turn, this implies that 
\begin{align*}
|\omega^m-\overline{\omega}^m| &=|\omega^n-\overline{\omega}^n| \\ 
& \leq 2|\omega|^n=2^{n/2+1} \\
& \leq 2^{\log(m)/ \log(4)+2},
\end{align*} or that 
\begin{equation} \label{upper_bound_omega^n-omega_bar}
\log |\omega^m-\overline{\omega}^m| \leq \frac{\log(m)}{2} +\log(4).
\end{equation}

Consider the algebraic number $$\alpha=\frac{\omega}{\overline{\omega}}=-\frac{3}{4}-i\frac{\sqrt{7}}{4}.$$ The minimal polynomial of $\alpha$ is $2x^2+3x+2$, so that 
$h(\alpha)=\log(2)/2$. We also note that $|\log(\alpha)| \approx 2.418858$, thus 
$10.98  |\log \alpha |+ h(\alpha)\approx 26.9056>20.$ Finally, $\log \left( \frac{m}{25} \right)+7.45>17$ for $m\geq 352000$. In this range, the results of Laurent, Mignotte and Nesterenko (see Remark~\ref{sharp_variant}) imply that 
$$
\log|\alpha^m-1| \geq -9 \cdot 26.9056  \left( \log \left( \frac{m}{25}\right)+7.45 \right)^2  > -2^9 (\log m)^2,$$
 and thus 
\begin{equation} \label{lower_bound_omega^n-omega_bar}
\log |\omega^m-\overline{\omega}^m |>\frac{m}{2} \log(2)-2^9 (\log m)^2. 
\end{equation}
Therefore, comparing (\ref{lower_bound_omega^n-omega_bar}) and (\ref{upper_bound_omega^n-omega_bar}), we obtain
$$\frac{m}{2} \log(2)-2^9 (\log m)^2 < \frac{\log(m)}{2} +\log(4),$$ which is false for $m>352000$. In other words, if a solution $11\leq n<m$ existed then $m<352000$. However, with the help of a computer it can be easily verified that there are also no solutions in this range. 
\end{proof}

\bigskip


\section{Repeated values in the sequences \texorpdfstring{$\{a_n\pm 2^n\}$}{a\_n+2^n}}

The main result of this section is the following.

\begin{proposition}
Consider the sequence $\{a_n\}$ defined in \eqref{def_a_n} and $\varepsilon\in\{\pm1\}$. The only pair of distinct positive integers $(n,m)$ such that
\begin{equation} \label{a_n-2^n}
a_n-\varepsilon 2^n=a_m- \varepsilon 2^m
\end{equation}
is $(5,9)$, and this occurs when $\varepsilon=-1$. In particular, as $k\geq 8$ ranges over the multiples of $4$, all the values of ${\rm Tr ~T_2}(2k,\SL_2(\Z))$ are distinct. Similarly, as $k\geq 9$ ranges over the integers $\equiv 1\Mod 4$, all the values of ${\rm Tr ~T_2}(2k,\SL_2(\Z))$ are distinct.
\end{proposition}

\begin{proof}
Suppose that (\ref{a_n-2^n}) has a solution with $n<m$. First, we will show that 
	\begin{equation} \label{n_less_log(m)}
		 n  \leq 3+ \frac{\log(m)}{ \log(2)}. 
	\end{equation}
As before, set $\omega=\frac{1}{2}-i\frac{\sqrt{7}}{2}$ so that $a_m=\frac{1}{\sqrt{-7}}(\omega^{2m+1}-\overline{\omega}^{2m+1} ).$ Equation (\ref{a_n-2^n}) implies
\begin{equation} \label{LHS_vs_RHS}
\varepsilon(2^m-2^n)+\frac{1}{\sqrt{-7}}(\omega^{2n+1}-\omega^{2m+1})=\frac{1}{\sqrt{-7}} \overline{\omega}^{2n+1}(1-\overline{\omega}^{2m-2n}).
\end{equation}
The left-hand side of (\ref{LHS_vs_RHS}) can also be expressed as
$$ 2^n \left( \varepsilon (2^{m-n}-1) +\frac{\omega}{\sqrt{-7}} \left( \frac{\omega^2}{2} \right)^n (1-\omega^{2m-2n} ) \right).$$ Since $v_2(\omega)=1$ and $v_2(\overline{\omega})=0$, we find that the 2-adic valuation of the above expression is 
\begin{align*}
& n+v_2 \left( \varepsilon (2^{m-n}-1)+\frac{\omega}{\sqrt{-7}} \left( \frac{\omega^2}{2} \right)^n (1-\omega^{2m-2n}) \right) \\ &  \geq n+ \min \left \{ v_2(2^{m-n}-1), v_2 \left( \frac{\omega}{\sqrt{-7}} \left( \frac{\omega}{\overline{\omega}} \right)^n (1-\omega^{2m-2n} ) \right) \right \} \\
&= n+\min \{ 0, 1+n \}=n.
\end{align*}
At the same time, the 2-adic valuation of the right-hand side of (\ref{LHS_vs_RHS}) is 
\begin{align*}
v_2 \left( 1-(\overline{\omega}^2)^{m-n} \right) & = v_2 \left( 1-(\overline{\omega}^2) \right) +v_2(m-n) \\
&= 3 +v_2(m-n) \\
& \leq 3 + \frac{\log(m)}{ \log(2)}. 
\end{align*}
Together, the previous two inequalities imply (\ref{n_less_log(m)}). Since $|\omega|^2=2$ and $|\omega-\overline{\omega}|=\sqrt{7}$, we obtain
\begin{align*}
|a_n-\varepsilon 2^n| &= \left | \frac{\omega}{ \omega -\overline{\omega}} (\omega^2)^n -
 \frac{\overline{\omega}}{ \omega -\overline{\omega}} ( \overline{\omega}^2)^n
-\varepsilon 2^n \right| \\
& < 3\cdot 2^n \\
& < 3\cdot 2^{3+\log(m)/\log(2)},
\end{align*}
or equivalently 
\begin{equation} \label{upper_bound_log|a_n-2^n|}
\log |a_m-\varepsilon 2^m| = \log |a_n-\varepsilon 2^n| \leq \log(24m).
\end{equation}
Our next step will be to find a lower bound for $\log |a_m-\varepsilon 2^m|$. Writing $w=\sqrt{2}e^{i\theta}$, with $\theta\in [\frac{3\pi}{2},2\pi]$, we obtain 
\begin{align} \label{sin((2m+1)theta)}
a_m-\varepsilon 2^m &= 2^m \left( \frac{\sin((2m+1)\theta)}{\sin \theta} -\varepsilon \right ) \notag \\
& =-2^m \sqrt{\frac{8}{7}} \left( \sin((2m+1)\theta) -\sin (\varepsilon \theta) \right),
\end{align}
since $\sin \theta=-\sqrt{\frac{7}{8}}$.

If $\varepsilon=1$ we see that
\begin{align*}
 | \sin((2m+1)\theta) -\sin \theta  | &= 2| \sin (m\theta) \cos ((m+1)\theta)| \\
 & =2 \left| \sin \left( m\theta -k_0\frac{\pi}{2} \right) \sin \left( (m+1)\theta -k_1\frac{\pi}{2} \right)\right|,
\end{align*} for all even integers $k_0$ and all odd integers $k_1$. We choose $k_0$ to be the even number than minimizes $ \left| m\theta -k_0 \frac{\pi}{2} \right|$, and $k_1$ the odd number that minimizes $ \left| (m+1)\theta -k_1 \frac{\pi}{2} \right|$. In this case, both of these quantities\footnote{We note that neither of these quantities can be zero. Indeed, if $\theta$ were a rational multiple of $\pi$, then both $e^{i\theta}$ and $e^{-i\theta}$ would be algebraic integers. So $2\cos \theta=1/\sqrt{2}$ would also be an algebraic integer, which is a contradiction.} are in the interval $\left( 0, \frac{\pi}{2} \right)$, so we can use the elementary inequality $|\sin(x)| \geq \frac{2}{\pi} |x|$, valid for all $|x|\leq \frac{\pi}{2}$. This gives 
\begin{align} \label{m_and_m_1_log}
 | \sin((2m+1)\theta) -\sin \theta  | & \geq 2 \left( \frac{2}{\pi} \right)^2 \left| m\theta -k_0 \frac{\pi}{2} \right| \left| (m+1)\theta -k_1 \frac{\pi}{2} \right| \notag \\ 
 &=\frac{2}{\pi^2} \left| 2m\theta - k_0 \pi \right|  \left| 2(m+1)\theta -k_1 \pi \right| \notag \\
 &=\frac{2}{\pi^2} \left| m \log e^{2i\theta} - k_0 i \pi \right| \left|(m+1) \log e^{2i\theta} - k_1 i \pi \right|. 
\end{align}
At this point we can apply again the result of Laurent, Mignotte and Nesterenko. In fact, since $e^{2i\theta}=\omega/\overline{\omega}$, we have as in Proposition~\ref{compare_a_n} that $10.98  |\log e^{2i\theta} |+ h(e^{2i\theta})>20.$ Then by Lemma~\ref{sharp_two_log} we obtain 
\begin{align*}
\left| m \log e^{2i\theta} - k_0 i \pi \right|  & \geq \exp(-8.87\cdot 26.9056\cdot H_1^2) \\
& > \exp(-2^8 \cdot H_1^2),
\end{align*}
where $$H_1=\max \left\{ 17, \log \left( \frac{k_0}{53.81} + \frac{m}{68.9} \right) +7.38 \right\}.$$ By construction, $k_0\leq m\theta/(\pi/2)+1< m\cdot 0.76995 +1$, and thus
$$ \log \left( \frac{k_0}{53.81} + \frac{m}{68.9} \right) +7.38 < \log (0.05\cdot m) + 7.38<\log(m)+4.385.$$
Now assume that $m\geq 302000$, so that $H_1<\log(m)+4.385$. In this case
\begin{equation} \label{lower_bound_k_0}
\left| m \log e^{2i\theta} - k_0 i \pi \right|  >\exp \left(-2^8 ( \log (m)+4.385)^2 \right).
\end{equation}
It is clear that the right-hand side of (\ref{lower_bound_k_0}) also serves as a lower bound for the second absolute value that appears in equation (\ref{m_and_m_1_log}). Therefore
$$  | \sin((2m+1)\theta) -\sin \theta  | \geq \exp \left(-2^9 (
\log (m)+4.385)^2 \right).$$
Similarly, if $\varepsilon=-1$ we obtain that
$$  | \sin((2m+1)\theta) -\sin (-\theta)  | \geq \exp \left(-2^9 (
\log (m)+4.385)^2 \right).$$
By (\ref{sin((2m+1)theta)}) it follows that for $m\geq 302000$:
\begin{equation} \label{lower_bound_|a_m-2^m|}
\log( |a_m-\varepsilon 2^m|) \geq m\log(2)+\log(\sqrt{8/7}) -2^9 ( \log (m)+4.385)^2. 
\end{equation}
Combining (\ref{upper_bound_log|a_n-2^n|}) and (\ref{lower_bound_|a_m-2^m|}) leads to the following inequality 
$$\log(24m) \geq m\log(2)+\log(\sqrt{8/7}) -2^9 ( \log (m)+4.385)^2,$$ which certainly does not hold for $m\geq 302000$. 
To finish the proof, we run a quick computer check confirming that for $m<302000$ there are no solutions to \eqref{a_n-2^n} in the case $\varepsilon=1$. 

When $\varepsilon=-1$ we do obtain one solution, namely $$a_5+2^5=a_9+2^9=55.$$
However, this does not yield an equality of traces because of the congruence conditions from \eqref{a(n)}. Indeed, ${\rm Tr ~T_2}(12,\SL_2(\Z))=-a_5-1=-24$, whereas ${\rm Tr ~T_2}(20,\SL_2(\Z))=-a_9-1=456$.

\end{proof}

\section*{Acknowledgement} We are grateful to Jeffrey Ovall for helpful conversations. 

\bigskip

\bibliographystyle{alpha}

\bibliography{ref}

\end{document}